%% file: BE14_GAFA.tex
\documentclass[12pt]{article}

\usepackage[colorlinks]{hyperref}            
\usepackage{color}
\usepackage{graphicx,subfigure,amsmath,amssymb,amsfonts,bm,epsfig,epsf,url,dsfont}
\usepackage{times}
\usepackage{bbm}      
\usepackage{booktabs}
\usepackage{cases}
\usepackage{fullpage}
\usepackage[small,bf]{caption}
\usepackage{natbib}
\usepackage[top=1in,bottom=1in,left=1in,right=1in]{geometry}
\usepackage{enumerate}

\include{Commands}

\def\eps{\varepsilon}
\newtheorem{theorem}{Theorem}
\newtheorem{lemma}{Lemma}
\newtheorem{definition}{Definition}
\newcommand{\BlackBox}{\rule{1.5ex}{1.5ex}}  
\newenvironment{proof}{\par\noindent{\bf Proof\ }}{\hfill\BlackBox\\[2mm]}

\begin{document}

\title{The entropic barrier: a simple and optimal universal self-concordant barrier}

\author{S\'ebastien Bubeck
	\thanks{Microsoft Research and Princeton University; \texttt{sebubeck@microsoft.com}.}
	\and
	Ronen Eldan 
	\thanks{Weizmann Institute of Science; \texttt{roneneldan@gmail.com}.}}
\date{\today}

\maketitle

\begin{abstract}
We prove that the Cram\'er transform of the uniform measure on a convex body in $\R^n$ is a $(1+o(1)) n$-self-concordant barrier, improving a seminal result of Nesterov and Nemirovski. This gives the first explicit construction of a universal barrier for convex bodies with optimal self-concordance parameter. The proof is based on basic geometry of log-concave distributions, and elementary duality in exponential families.
\end{abstract}

\section{Introduction}
Let $\cK \subset \R^n$ be a convex body, namely a compact convex set with a non-empty interior. Our main result is:
\begin{theorem} \label{th:main}
Let $f: \R^n \to \R$ be defined for $\theta \in \R^n$ by
\begin{equation} \label{eq:deff}
f(\theta) = \log \left(\int_{x \in \cK} \exp(\langle \theta, x \rangle) dx \right) .
\end{equation}
Then the Fenchel dual $f^* : \mathrm{int}(\cK) \to \R$, defined for $x \in \inte(\cK)$ by $f^*(x) = \sup_{\theta \in \R^n} \langle \theta, x \rangle - f(\theta)$, is a $(1+\epsilon_n) n$-self-concordant barrier on $\cK$, with $\epsilon_n \leq 100 \sqrt{\log(n)/n}$, for any $n \geq 80$. 
\end{theorem}
In Section \ref{sec:relatedwork} we recall the definition of a $\nu$-self-concordant barrier and its importance in mathematical optimization. We give another point of view on $f^*$ in Section \ref{sec:expfamily}, where we show that it corresponds to the negative entropy of a specific element in a canonical exponential family for $\cK$. For this reason we refer to $f^*$ as the {\em entropic barrier} for $\cK$. Finally, we prove Theorem \ref{th:main} in Section \ref{sec:proof}. Technical lemmas on log-concave distributions are gathered in Section \ref{sec:tech}, where in particular we derive the sharp bound $\E X^3 \leq 2$ for a real isotropic log-concave random variable $X$.

\section{Context and related work} \label{sec:relatedwork}
For a $C^3$-smooth function $g:\R^n \to \R$, denote by $\nabla^2 g [\cdot, \cdot]$ its Hessian which we understand as a bilinear form over $\R^n$. Likewise, by $\nabla^3 g [\cdot, \cdot, \cdot]$ we denote its third derivative tensor. We first recall the definition, introduced in \cite{NN94}, of a self-concordant barrier. 

\begin{definition}
A function $g: \inte(\cK) \rightarrow \R$ is a barrier for $\cK$ if
$$g(x) \xrightarrow[x \to \partial \cK]{} +\infty .$$
A $C^3$-smooth convex function $g : \inte(\cK) \rightarrow \R$ is self-concordant if for all $x \in \inte(\cK), h \in \R^n$,
\begin{equation} \label{eq:sc}
\nabla^3 g(x) [h,h,h] \leq 2 (\nabla^2 g(x)[h,h])^{3/2} .
\end{equation}
Furthermore it is $\nu$-self-concordant if in addition for all $x \in \inte(\cK), h \in \R^n$,
\begin{equation} \label{eq:nusc}
\nabla g(x)[h] \leq \sqrt{\nu \cdot \nabla^2 g(x)[h,h] } .
\end{equation}
\end{definition}
Self-concordant barriers are central objects in the theory of Interior Point Methods (IPMs). The latter class of algorithms has revolutionized mathematical optimization, starting with \cite{Kar84}. Roughly speaking, an IPM minimizes the linear function $x \in \cK \mapsto \langle c, x \rangle$ (for some given $c \in \R^n$) by tracing the {\em central path} $(x(t))_{t \in (0, +\infty)}$ of a self-concordant barrier $g$ for $\cK$, where $x(t) \in \argmin_x \langle c, x \rangle + \frac{1}{t} g(x)$. The key property of $\nu$-self-concordant barriers is that a step of Newton's method on the function $x\mapsto \langle c, x \rangle + \frac{1}{t} g(x)$ allows to move from $x((1-1/\sqrt{\nu}) t)$ to (approximately) $x(t)$, see e.g. \cite{Nes04} for more details. In other words in $O(\sqrt{\nu})$ steps of Newton's method on $g$ one can approximately minimize a linear function on $\cK$. 

From a theoretical point of view, one of the most important results in the theory of IPM is Nesterov and Nemirovski's construction of the {\em universal barrier}, which is a $\nu$-self-concordant barrier that always satisfies $\nu \leq C n$, for some universal constant $C>0$. Theorem \ref{th:main} is the first improvement (for convex bodies) over this seminal result: we show that in fact there always exists a barrier with self-concordance parameter $\nu = (1+o(1))n$. 
Up to the second-order term, this improved self-concordance parameter is also optimal, as one must have $\nu \geq n$ for some convex sets (such as a simplex or a hypercube, see [Proposition 2.3.6., \cite{NN94}]). In fact, as we explain next, we can prove that there always exists a barrier with self-concordance parameter $\nu = n+1$ (and $\nu = n$ for convex cones\footnote{Daniel Fox pointed out to us that this bound had been conjectured by Osman G{\"u}ler since the mid-Nineties.}). This result was also independently obtained in \cite{Hil14, Fox15} with a different construction refered to as the {\em canonical barrier}. Thanks to its probabilistic interpretation the entropic barrier is much easier to analyze than the canonical barrier, and it can also be applied to problems where the canonical barrier or the universal barrier would give suboptimal results (an example is given in Section \ref{sec:linearbandit}). We discuss the connections between these three barriers in more details in Section \ref{sec:threebarriers}.

From a convex geometry point of view the entropic barrier is a natural object. Indeed, as demonstrated in several recent works, the log-Laplace transform is a useful tool in proving inequalities related to high dimensional convex bodies. As a cumulant-generating function associated with a given convex body, it provides an analytical viewpoint which is a central theme in proving bounds related to its distribution of mass (see e.g. \cite{Kla06,KM2011,EK11}). These bounds often boil down to proving relations satisfied by its derivatives, which are often of the same spirit as \eqref{eq:sc}. It seems conceivable that a better understanding of the entropic barrier may also be useful for proving such inequalities. Curiously, a function closely related to the canonical barrier was also implicitly used in \cite{Kla14} to prove inequalities of the same spirit.

It is interesting to observe that the self-concordance property (i.e., \eqref{eq:sc}) is dimension-free, while on the other hand the dimension plays a key role in inequality \eqref{eq:nusc}. For the entropic barrier the latter inequality is in a sense more delicate in that it extracts the fact that the marginal of an $n$-dimensional convex body has a strictly better behavior than mere log-concavity. It is somewhat mysterious that the effect of the dimension on the complexity of Interior Point Methods can be derived as a consequence of such a delicate local behavior.

\subsection{Universal, canonical, and entropic barriers for convex cones} \label{sec:threebarriers}
For a convex set $\cK$, we denote by $\cK^{\circ} = \{y \in \R^n : \sup_{x \in \cK} \langle y, x \rangle \leq 1\}$ the polar of $\cK$, and by $\cK^* = \{y \in \R^n : \inf_{x \in \cK} \langle y, x \rangle \leq 0 \}$ the dual cone of $\cK$ (note the sign difference with respect to the standard definition of a dual cone). Nesterov and Nemirovski's universal barrier is defined as follows, for any $x \in \cK$,
$$u(x) = \log \mathrm{vol} \left( (\cK - x)^{\circ} \right) .$$
An important point in the theory of Interior Point Methods is the so-called {\em analytical center} of $\cK$, defined as the minimizer of a self-concordant barrier. As a side note we observe that the analytical center for the universal barrier corresponds to the well-known Santal\'o point in convex geometry.

For the rest of this section $\cK$ denotes a proper convex cone. A function $g : \cK \to \R$ that satisfies $g(t x) = g(x) - \nu \log t$ for any $x \in \cK, t >0$ is said to be $\nu$-logarithmically homogeneous. It is well known that a $C^2$-smooth $\nu$-logarithmically homogeneous function satisfies \eqref{eq:nusc}, and furthermore if $g$ is $\nu$-self-concordant on $\cK$ then its Fenchel dual $g^*$ is $\nu$-self-concordant on the dual cone $\cK^*$ (see \cite{NN94}). We also recall that the characteristic function $\phi_{\cK}$ of $\cK$ is defined by, for any $x \in \cK$,
$$\phi_{\cK}(x) = \int_{\cK^*} \exp(\langle \theta, x \rangle ) d\theta .$$
In \cite{Gul96} it is shown that $u(x) = \log n! + \log \phi(x)$. Thus it is immediate that the universal barrier is $n$-logarithmically homogeneous, and using the argument we develop at the beginning of Section \ref{sec:proof} (i.e., the connection with moments of a log-concave distribution together with Lemma \ref{lem:32}) one can show that the universal barrier also satisfies \eqref{eq:sc}. In other words we obtain the following new result: the universal barrier is $n$-self-concordant on convex cones. This also implies that for any convex body there exists a $(n+1)$-self-concordant barrier, simply by considering the universal barrier on the conic hull of the convex body. Note that this latter construction is different from considering the universal barrier of the convex body itself (which is $C n$-self-concordant, for some numerical constant $C$, as proven by Nesterov and Nemirovski).

The definition of the entropic barrier given in Theorem \ref{th:main} does not directly apply to convex cones, in the sense that the Laplace transform $f$ is only defined for $\theta \in \cK^*$ (instead of $\theta \in \R^n$). Thus we naturally define the entropic barrier $e$ on a convex cone as the Fenchel dual of the logarithm of the characteristic function on the dual cone, that is $e(x)=(\log \phi_{\cK^*} )^*(x)$. In particular the entropic barrier on a cone is the Fenchel dual of the universal barrier on the dual cone, and thus it is also a $n$-self-concordant barrier.

The third barrier for convex cones with optimal self-concordance parameter is the canonical barrier $c$, introduced in \cite{Hil14, Fox15}. It is defined as the unique convex solution to the Monge-Amp\`ere equation $2 f = \log \mathrm{det} \nabla^2 f$ with the boundary condition $f|_{\partial \cK} = +\infty$. This equation is known to exhibit a unique solution on convex cones, but is generally not solvable on (compact) convex bodies. The Riemannian metric induced by the canonical barrier is also known as the K{\"a}hler-Einstein metric.  To gain some intuition in the defining equation of the canonical barrier, observe that $- \frac12 \mathrm{det} \nabla^2 f (x)$ is the log-volume of the {\em Dikin ellipsoid} $\cE_f(x)$ of $f$ at $x$ (recall that $\cE_f(x):=\{y \in \R^n : \nabla^2 f(x)[y-x,y-x] \leq 1\}$ and that for $f$ self-concordant one has $\forall x \in \mathrm{int}(\cK), \cE_f(x) \subset \cK$). In other words $c$ is the unique function whose value at point is given by the logarithm of the inverse volume of the Dikin ellipsoid at this point.

Perhaps surprisingly, \cite{Gul96} showed that in the case of homogeneous convex cones the three above barriers coincide (up to a constant). This connection is nontrivial, and somewhat mysterious to us. More generally the relations between the universal, canonical and entropic barriers remain quite elusive. For example, another interpretation of the canonical barrier is that $c(x)$ is equal to the differential entropy of a centered Gaussian with covariance given by $\nabla^2 c(x)$, while as we shall see in Section \ref{sec:expfamily} one has that $e(x)$ is equal to the differential entropy of a natural log-affine distribution (supported on $\cK$) whose covariance matrix is given by $\nabla^2 e(x)$.

We conclude this discussion with a comment on the importance of cones. It is well-known that for convex optimization one can assume without loss of generality that $\cK$ is a convex cone, which is why several authors focused on this case. However there are other applications of the theory of self-concordant barriers where it is important to have a (tractable) barrier for convex bodies too, such as in the sampling problem \cite{KN12}. In Section \ref{sec:linearbandit} we briefly describe another such application to an online learning problem, where in addition the probabilistic interpretation of the entropic barrier is essential (the described result cannot be obtained with the universal barrier or the canonical barrier).

\subsection{Computational aspects}
It is important to note that the universal, canonical, and entropic barriers are not (immediately at least) relevant in practice. Indeed, the computational effort to implement an IPM depends on the complexity of calculating gradients and Hessians for the barrier. The key to the practical success of IPM is that for important classes of convex sets there exist self-concordant barriers with efficiently computable gradients and Hessians. 
While this is certainly not immediately the case for the entropic barrier, there is some hope: for instance, its inverse Hessian corresponds to the covariance matrix of a simply described log-concave distribution (a similar statement is true for the universal barrier, but the distribution is more complicated to describe). Furthermore, it can be seen that given a membership oracle to $\cK$, there exists a randomized algorithm which approximates the value of this barrier at a given point in polynomial time. This can be done by sampling from the distribution $p_\theta$ (defined below) via standard techniques (see e.g., \cite{LV07}).


Finally we note that even in the simplest situation where $\cK$ is a polytope, it remained open until very recently (\cite{LS14}) to find an efficiently computable barrier with self-concordance parameter nearly matching the one of the universal barrier. We hope that our new barrier will help making progress in finding efficient and optimal barriers.

\section{A canonical exponential family} \label{sec:expfamily}
In this section we introduce and briefly study the canonical exponential family $\{p_{\theta}, \theta \in \R^n\}$ associated with $\cK$. For $\theta \in \R^n$, let $p_{\theta}$ be the probability measure on $\R^n$ whose density with respect to the Lebesgue measure at $x \in \R^n$ is
$$\exp(\langle \theta, x \rangle - f(\theta)) \ds1\{x \in \cK\} ,$$
where $f$ is as in \eqref{eq:deff}. In other words $f$ is the log-partition function for this exponential family. We denote $x(\theta) := \E_{X \sim p_{\theta}} X$. It is well-known (see e.g., [Section 3, \cite{Kla06}]) that $\theta \mapsto x(\theta)$ is a bijection between $\R^n$ and $\inte(\cK)$ (we denote $x \in \inte(\cK) \mapsto \theta(x)$ for the inverse mapping, which is onto $\R^n$), and that $f$ is strictly convex, $C^{\infty}$-smooth, and $\nabla f(\theta) = x(\theta)$. With these observations it is an elementary calculation to recover  a basic duality result for exponential families (see e.g. [Theorem 3.4., \cite{WJ08}]), namely that $f^*(x) = - H(p_{\theta(x)})$, where $H(p)$ is the differential entropy of $p$, 
defined by
$$
H(p) := - \int_{\R^n} p(x) \log p(x) dx.
$$
Hence the name {\em entropic barrier} for $f^*$. Recall also that $\nabla f^*(x) = \theta(x)$.

We will also need higher moments of $f$ and $f^*$. Let $\Sigma(\theta) := \E_{X \sim p_{\theta}} (X-x(\theta)) (X-x(\theta))^{\top}$, and $T(\theta) := \E_{X \sim p_{\theta}} (X-x(\theta)) \otimes (X-x(\theta)) \otimes (X-x(\theta))$. It is again an easy exercise (partly done in \cite{Kla06}) to show that $\nabla^2 f(\theta) = \Sigma(\theta)$, $\nabla^3 f(\theta) = T(\theta)$ and $\nabla^2 f^*(x) = \Sigma(\theta(x))^{-1}$ (see for example [(2.15), \cite{Nem04b}] for the latter equality). \\

We summarize the above in a lemma.
\begin{lemma}
The functions $f,f^*$ satisfy the following.
\begin{enumerate}[(i)]
\item The function $f$ is strictly convex on $\R^n$ and the function $f^*$ is strictly convex in the interior of $\mathcal{K}$.
\item The function $\theta(\cdot) = \nabla f^*(\cdot)$ is a bijection between the interior of $\mathcal{K}$ and $\R^n$.
\item One has for all $\theta \in \R^n$,
\begin{equation} \label{eq:secondder}
\nabla^2 f(\theta) = \E_{X \sim p_{\theta}} (X-x(\theta)) (X-x(\theta))^{\top} = \Sigma(\theta).
\end{equation}
and
\begin{equation} \label{eq:thirdder}
\nabla^3 f(\theta) = \E_{X \sim p_{\theta}} (X-x(\theta)) \otimes (X-x(\theta)) \otimes (X-x(\theta)) = T(\theta).
\end{equation}
\item One has for all $x \in \mathrm{int}(\mathcal{K})$,
\begin{equation} \label{eq:secondder2}
\nabla^2 f^*(x) = \left (\nabla^2 f(\theta(x))\right )^{-1} = \left (\E_{X \sim p_{\theta(x)}} (X-x) (X-x)^{\top} \right )^{-1} = \Sigma(\theta(x))^{-1}.
\end{equation}
\end{enumerate}
\end{lemma}

Next we describe an application where the connection between the entropic barrier and the canonical exponential family $\{p_{\theta}\}$ is crucial.

\subsection{An application to the linear bandit problem} \label{sec:linearbandit}
We consider a sequential extension of linear optimization, known as {\em online linear optimization}. It can be described as the following sequential game: at each time step $t=1, \hdots, T$, a player selects an action $x_t \in \cK$, and simultaneously an adversary selects a cost vector $c_t \in \cK^{\circ}$ (where $\cK^{\circ}$ is the polar of $\cK$). Both the action and the cost are selected as a function of the history $(x_s, c_s)_{s < t}$, and possibly external randomness (independent for the player and the adversary). The player's perfomance at the end of the game is measured through the {\em regret}:
$$R_T = \sum_{t=1}^T \langle c_t, x_t\rangle - \min_{x \in \cK} \sum_{t=1}^T \langle c_t, x \rangle ,$$
which compares her cumulative cost to the best cumulative cost she could have obtained in hindsight with a fixed action, if she had known the sequence of costs played by the adversary. This problem has a long history, and a wealth of applications, see, e.g., \cite{CL06}. A far more challenging scenario is when the player only receives a limited feedback on the cost function. Of particular interest is the {\em bandit feedback}, where the player only observes her incurred cost $\langle c_t, x_t\rangle \in \R$, rather than the full cost vector $c_t \in \R^n$. See \cite{BC12} for a recent survey on bandit problems. In the following we show how Theorem \ref{th:main} gives a new point of view on some known results for online linear optimization with bandit feedback.

Since the seminal work of \cite{AHR08} it is known that self-concordant barriers play an important role in the design of good player's strategies. More precisely the latter paper proposed to run {\em Mirror Descent} (which was originally introduced in \cite{NY83}) with a self-concordant barrier as the mirror map. In addition to the choice of a barrier, one also needs to choose a sampling scheme, that is a mapping from actions to distributions over actions. A key insight of \cite{AHR08} is that the barrier and the sampling scheme should ``match'' each other, in the sense that the Hessian of the barrier should be approximately proportional to the inverse covariance of the sampling scheme. In \cite{AHR08} this is achieved with a sampling scheme supported on the Dikin ellipsoid, and they prove that with the universal barrier this yields $\E R_T = O(n^{3/2} \sqrt{T \log T})$. By using the entropic barrier together with the sampling scheme $x \mapsto p_{\theta(x)}$, it is easy to see that one can improve the bound to $\E R_T = O(n \sqrt{T \log T})$, thus matching the state of the art bound of \cite{BCK12} (which, up to the logarithmic factor, is the best possible universal bound). The improvement over \cite{AHR08} is due to the fact that the sampling scheme $p_{\theta(x)}$ makes a much better use of the available ``space'' around $x$ than the Dikin's ellipsoid sampling. We also note that \cite{BCK12} obtained their bound via exponential weights on a discretization of $\cK$, while it is easy to see that Mirror Descent with the entropic barrier and its associated sampling scheme exactly corresponds to continuous exponential weights, a strategy introduced in \cite{Cov91} for the full information case. In both cases one has to use the John exploration described in \cite{BCK12} to obtain the bound mentioned above, though one can envisage more efficient alternatives such as those described in \cite{HKM14}.
%
%
%
%
%
%
%

\section{Proof of Theorem \ref{th:main}} \label{sec:proof}
Since $\nabla f(\R^n) = \inte(\cK)$, a basic property of the Fenchel transform is that $f^*$ is a barrier for $\cK$. Next we show that $f^*$ is self-concordant on $\cK$ by proving that $f$ is self-concordant on $\R^n$ (the implication then follows from [Section 2.2., \cite{Nem04b}]). By definition, and using equation \eqref{eq:secondder} and \eqref{eq:thirdder}, $f$ is self-concordant if for any $\theta, h \in \R^n$,
$$\E_{X \sim p_{\theta}} \langle X-x(\theta) , h \rangle^3 \leq 2 \left(\E_{X \sim p_{\theta}} \langle X-x(\theta) , h \rangle^2 \right)^{3/2} .$$
Noting that $p_{\theta}$ is a log-concave measure one immediately obtains the above equation with a worse numerical constant from [(2.21), \cite{Led01}]. The numerical constant $2$ can be obtained via the following lemma, whose proof can be found in Section \ref{sec:tech}.
\begin{lemma} \label{lem:32}
Let $X$ be a real log-concave and centered random variable. Then
$$\E X^3 \leq 2 \left(\E X^2\right)^{3/2}.$$
\end{lemma}

We now move to the main part of the proof, which is to bound the self-concordance parameter $\nu$ of $f^*$ by $(1+\epsilon_n)n$. By setting $h= (\nabla^2 f^{*}(x))^{-1/2} w$, equation \eqref{eq:nusc} becomes
$$
\left \langle (\nabla^2 f^{*}(x))^{-1/2} \nabla f^{*}(x), w \right \rangle  \leq \sqrt{\nu \langle w,w \rangle} .
$$
and therefore, \eqref{eq:nusc} is equivalent to
$$\left \langle \left(\nabla^2 f^{*}(x)\right)^{-1} \nabla f^{*}(x), \nabla f^{*}(x) \right \rangle \leq \nu .$$
Thus, according to equation \eqref{eq:secondder2}, we have to show that for any $\theta \in \R^n$,
$$\langle \Sigma(\theta) \theta, \theta \rangle \leq (1+\epsilon_n) n .$$
In other words, considering the random variable $Y = \left \langle \frac{\theta}{\|\theta\|}, X \right \rangle$, with $X \sim p_\theta$, the proof will be concluded by showing that
\begin{equation} \label{eq:toshow}
\mathrm{Var}(Y) \leq \frac{n}{\|\theta\|^2}(1+\epsilon_n) .
\end{equation}
We denote by $\rho$ the density of $Y$, which is proportional to
\begin{equation} \label{eq:marginaldensity}
\mathrm{Vol}_{n-1}\left(\cK \cap \{y \theta / \|\theta\| + \theta^{\perp} \} \right) \exp( y \|\theta\| ) .
\end{equation}
At this point we observe that, without loss of generality, we can assume that $\rho$ is a $C^{\infty}$-smooth function in the interior of its support. Indeed, consider a sequence $\cK_1 \subset \cK_2  \subset ...$ of convex bodies with a $C^{\infty}$-smooth boundary which satisfy $\bigcup_{k} \cK_k = \cK$, and let $p_\theta^{(k)}$ be the canonical exponential family associated with $\cK_k$. Then for all $\theta$ we have that $p_\theta^{(k)}$ converges weakly to $p_\theta$, which implies that the covariance matrix of $p_\theta^{(k)}$ converges (in operator norm) to $\Sigma(\theta)$ as $k \to \infty$. Therefore, it is enough to verify equation \eqref{eq:toshow} for $p_\theta^{(k)}$. By the smoothness and compactness of $\cK_k$, the marginal $\rho$ will be a smooth function on its support.

The most technical step of the proof is the following lemma, which relies on the log-concavity properties of the $1$-dimensional marginals of the uniform measure on $\cK$, and which states that $Y$ is "locally" sub-Gaussian. We give a proof of this lemma at the end of this section.

\begin{lemma} \label{lem:key}
Let $y_0 \in \argmax_{y \in \R} \rho(y)$, $M = \frac{\sqrt{7 n \log(n)}}{\|\theta\|}$, and $\sigma^2 = \frac{n}{\|\theta\|^2} \frac{1}{1 - \sqrt{7 \log(n) / n}}$. 
There exists $\zeta : [-M, M] \rightarrow [0, 1]$, increasing on $[-M,0]$, decreasing on $[0,M]$, and with $\zeta(0) = 1$, such that for any $y \in [-M,M]$,
$$\rho(y + y_0) = \rho(y_0) \zeta(y) \exp\left( - \frac{y^2}{2 \sigma^2} \right) .$$
\end{lemma}

The above lemma implies that, conditionally on $|Y-y_0| \leq M$, the random variable $|Y-y_0|$ is stochastically dominated by $|\cN(0, \sigma^2)|$. Indeed, the density of $|Y|$, conditioned on $|Y-y_0| \leq M$, with respect to the law of $|\cN(0, \sigma^2)|$ is equal to $q(y) := Z (\zeta(y) + \zeta(-y)) \mathbf{1}_{y<M}$ for a normalization constant $Z$. Since $(\zeta(y) + \zeta(-y))$ is non-increasing, we learn that there exists $t>0$ such that $q(y) \geq 1$ for $y \in [0,t]$ and $q(y) \leq 1$ for $y > t$ which confirms the assertion. 

This implies in particular
\begin{equation} \label{eq:sigma}
\E\left(|Y-y_0|^2  \ | \ |Y-y_0| \leq M \right) \leq \sigma^2 .
\end{equation}
It remains to show that the above conditional variance bound implies \eqref{eq:toshow}. For this we use another technical result, whose proof can be found in Section \ref{sec:tech}:
\begin{lemma} \label{lem:tail}
Let $\eps > 0$, and $X$ a real log-concave random variable with density $\lambda$. Let $x_1 < x_0 < x_2$ be three points satisfying $\lambda(x_1) < \eps \lambda(x_0)$ and $\lambda(x_2) < \eps \lambda(x_0)$. Then, with $c(\epsilon)= \left(1+ \frac{2}{\log(1/\epsilon)} \right)^3 \left(1+ \frac{2}{\log(1/\epsilon)} + \frac{2}{\log^2(1/\epsilon)}\right)$, one has
\begin{eqnarray*}
\left (1 - 2 c(\epsilon) \epsilon \log^2(1/\epsilon) \right ) \mathrm{Var}(X) & \leq & \int_{x_1}^{x_2} (x-x_0)^2 \lambda(x) dx  \\
& \leq & \E( |X-x_0|^2 \ | \ X \in [x_1, x_2]) .
\end{eqnarray*}
\end{lemma}
Thanks to Lemma \ref{lem:key}, we know that 
$$\max(\rho(y_0 - M), \rho(y_0+M)) \leq \left( \frac{1}{n} \right)^{\frac{7}{2} \left(1 - \sqrt{\frac{7 \log(n)}{n}} \right)} \rho(y_0) ,$$
and thus Lemma \ref{lem:tail} together with \eqref{eq:sigma} imply that, with $\epsilon=\left( \frac{1}{n} \right)^{\frac{7}{2} \left(1 - \sqrt{\frac{7 \log(n)}{n}} \right)}$,
$$\left (1 - \frac{70}{n} \right) \mathrm{Var}(Y) \leq \sigma^2 ,$$
which proves \eqref{eq:toshow}.

We now conclude the proof of Theorem \ref{th:main} with the proof of Lemma \ref{lem:key}.
\begin{proof}
Let $\lambda$ be the $1$-dimension marginal of the uniform measure on $\cK$ in the direction $\theta / \|\theta\|$, that is for $y \in \R$,
$$\lambda(y) = \frac{\mathrm{Vol}_{n-1}\left(\cK \cap \{y \theta / \|\theta\| + \theta^{\perp} \} \right)}{\mathrm{Vol}(\cK)} .$$
We already observed in \eqref{eq:marginaldensity} that
$$\rho(y) = \frac{\lambda(y) \exp(y \|\theta\|)}{\int_{\R} \lambda(s) \exp(s \|\theta\|) ds} .$$
It will be useful to consider the functions $u(y) = \log \lambda(y)$ and 
$$v(y) = \log \rho(y) = u(y) + y \|\theta\| - \log \left( \int_{\R} \lambda(s) \exp(s \|\theta\|) ds \right) .$$
Since we assumed that $\cK$ is a $C^{\infty}$-smooth domain, it is clear that $\lambda$ and $\rho$ are also $C^{\infty}$ on the interior of their support $[a,b]$, where
$$a = \inf \{s \in \R : \lambda(s) > 0 \}, ~~ b = \sup \{s \in \R : \lambda(s) > 0\}.$$
The key observation is that, thanks to the Brunn-Minkowski inequality, $\lambda$ is $n$-concave on its suppport (see, e.g., \cite{Bor75}), or in other words $\lambda^{1/n}$ is a concave function on the interval $(a,b)$. We now obtain a simple differential inequality by using the following lemma, whose proof can be found in Section \ref{sec:tech}:
\begin{lemma} \label{lem:simple}
Let $\phi \in C^2((a,b))$, and $\zeta(x) = \log \phi(x)$. Then 
$$
\phi \mbox{ is } n\mbox{-concave in } (a,b) \Leftrightarrow \zeta'' \leq - \tfrac 1 n (\zeta')^2 \mbox{ in } (a,b).
$$
\end{lemma}
The result of the lemma directly yields
$$ v''(y) = u''(y) \leq - \frac1{n} (u'(y))^2 = - \frac1{n} (v'(y) - \|\theta\|)^2 .$$
It is useful to rewrite the above inequality in terms of $w(y) := v'(y+y_0) + \frac{\|\theta\|^2 y}{n - \|\theta\| y}$, $y \in (a', b')$, with $a'=a-y_0, b'=\min(b-y_0, n/\|\theta\|)$, in which case one easily obtains
$$- w'(y) \geq \frac1{n} \left(w(y)^2 - 2 w(y) \frac{\|\theta\| n}{n - \|\theta\| y} \right).$$
Observe that $w(0) = 0$, as $y_0$ is a local maximum of the smooth function $v$. A simple application of Gronwall's inequality teaches us that $w$ is non-positive in the interval $[0, b')$ and non-negative in the interval $(a',0]$. We conclude that, for $y \in (a',b')$,
\begin{eqnarray*}
v(y+y_0) - v(y_0) & = & \int_{0}^{y} v'(s+y_0) ds \\
& = & - \int_{0}^{y} \frac{\|\theta\|^2 s}{n - \|\theta\| s} ds + \int_{0}^{y} w(s) ds \\ 
& = & \|\theta\|y + n \log(1 - \|\theta\| y / n) + \int_{0}^{y} w(s) ds \\
& = & - \left(1 - \sqrt{\frac{7 \log(n)}{n}} \right) \frac{y^2}{2 n /\|\theta\|^2} + n h\left(\frac{\|\theta\| y}{n}, \sqrt{\frac{7 \log(n)}{n}}\right) + \int_{0}^{y} w(s) ds ,
\end{eqnarray*}
where $h(x, \epsilon) = x + (1-\epsilon) \frac{x^2}{2} + \log(1-x)$. For any $\epsilon \in (0,1)$, $x \mapsto h(x, \epsilon)$ is increasing on $[- \epsilon/(1-\epsilon), 0]$ and decreasing on $[0,1)$. In particular, denoting $a''= \min(a', \sqrt{7 n \log(n)} / \|\theta\|)$ and $\sigma^2$ as in the statement of the lemma, we showed that there exists a function $\xi : (a'', b') \rightarrow \R_-$, increasing on $(a'',0)$, decreasing on $[0, b')$, with $\xi(0) = 0$, and such that
$$v(y+y_0) - v(y_0) = - \frac{y^2}{2 \sigma^2} + \xi(y) .$$
This easily concludes the proof.
\end{proof}

\section{Technical lemmas} \label{sec:tech}
In this section we prove Lemma \ref{lem:32}, \ref{lem:tail}, \ref{lem:simple} that were used in the preceding section to prove Theorem \ref{th:main}. In each case we first restate the lemma before going into the proof.

\setcounter{lemma}{1}
\begin{lemma}
Let $X$ be a real log-concave and centered random variable. Then
$$\E X^3 \leq 2 \left(\E X^2\right)^{3/2}.$$
\end{lemma}

\begin{proof}
We may assume that $X$ is supported in a compact interval $[-M,M]$. Indeed, if that is not the case, we can define $X_k$ to have the law of $X$ conditioned on the interval $[-k,k]$; if the result for the lemma holds for every $X_k$, we may take limits and deduce its correctness for $X$.

Define $g(x) = 1-x^2$. Denote by $P_g$ the family of log-concave probability measures $\mu$ on $\mathbb{R}$, supported on $[-M,M]$ which satisfy $\int g d \mu \geq 0$. Let $\Psi:P_g \to \mathbb{R}$ be a convex functional. According to \cite[Theorem 1, Theorem 2]{FG04}, which characterizes the extremal points of the set $P_g$, the supremum $\sup_{\mu \in P_g} \Psi(\mu)$ is attained either on a Dirac measure $\delta_x$ for some $x \in \mathbb{R}$ or on a measure $\mu$, which satisfies the following: \\
(i) The density of $\mu$ is log-affine on its support, hence there exist $a,b \in [-M,M]$,  $a<b$ and $c \in \mathbb{R}$ such that
\begin{equation} \label{eq:loglinear}
\frac{d \mu}{dx} = Z^{-1} \mathbf{1}_{[a,b]} e^{cx} ,
\end{equation}
where $Z$ is chosen such that $\mu$ is a probability measure. \\
(ii) One has $\int g d \mu = 0$. \\
(iii) There is $\xi \in \{-1,1\}$ such that for all $x \in [a,b]$ one has $\xi \int_a^{x} g d \mu \geq 0$. \\ \\

Consider the linear functional $\Psi(\mu) = \int (x^3-3x) d \mu$. It is clear that the statement of the lemma will follow if we establish that
\begin{equation} \label{eq:ntsapdx1}
\sup_{\mu \in P_g} \Psi(\mu) \leq 2.
\end{equation}
(note that for centered measures $\mu$, $\Psi(\mu)$ is exactly $\int x^3 d \mu(x)$). Now, it is clear that if $\mu$ is a Dirac measure satisfying $\int g d \mu \geq 0$, we have that 
$$
\Psi(\mu) \leq \sup_{x \in [-1,1]} (x^3-3x) = 2.
$$
Therefore, according to the above, it is enough to verify that the bound of equation \eqref{eq:ntsapdx1} holds for measures of the form \eqref{eq:loglinear}, which satisfy the constraint
\begin{equation} \label{eq:rhom2}
\int x^2 d \mu(x) = 1 ,
\end{equation}
and such that either $-1 \leq a \leq 1$ or $-1 \leq b \leq 1$ (this is a direct consequence of (iii) above). 

In other words, the lemma will be established if we show that
\begin{equation} \label{eq:ntsapx2}
H(a,b,c) := Z(a,b,c)^{-1} \int_a^b (x^3-3x) e^{cx} dx \leq 2
\end{equation}
for all $a,b,c$ such that either $|a|\leq 1$ or $|b| \leq 1$ and such that
\begin{equation} \label{eq:constraint1}
G(a,b,c) := Z(a,b,c)^{-1} \int_a^b (x^2 -1) e^{cx} dx = 0. 
\end{equation}
We will continue the proof under the assumption that $|a| \leq 1$. The proof under the assumption $|b| \leq 1$ is similar.

A calculation gives
\begin{equation} \label{eq:defZ}
Z(a,b,c) = \int_a^b e^{cx} dx = \frac{e^{cb}-e^{ca}}{c} ,
\end{equation}
and by integrating by parts, we have that
\begin{align} \label{eq:eqH2}
H(a,b,c) ~&= \frac{1}{c Z(a,b,c)} (x^3-3x) e^{cx}|_a^b - \frac{1}{cZ(a,b,c)} \int 3(x^2-1) e^{cx} dx \\
&\stackrel{\eqref{eq:defZ}, \eqref{eq:constraint1} }{=} \frac{(b^3-3b) e^{cb} - (a^3-3a) e^{ca}}{e^{cb}-e^{ca}}. \nonumber
\end{align}
It will be useful to set $r=e^{c(b-a)}$ and, respectively, define $c(a,b,r) = \log(r)/(b-a)$. 

Next, note that for any $a \in [-1,1]$ and any $r > 0$, one has $G(a,1,c(a,1,r)) \leq 0$. Moreover, observe that $G(a,b,c(a,b,r))$ is continuously increasing with respect to $b$ when $b \geq 1$. Thus, by the intermediate value theorem, there exists a unique point $b>a$ such that $G(a,b,c(a,b,r)) = 0$. Let us try to obtain a formula for this point. A straightforward calculation yields
$$
G(a,b,c) = \frac{(b c (b c-2)+2-c^2) e^{b c}-(a c (a c-2)+2-c^2) e^{a c}}{c^2 \left(e^{b c}-e^{a c}\right)} ,
$$
which means that the constraint $G(a,b,c) = 0$ is equivalent to 
$$
r = \frac{a c (a c-2)+2-c^2}{b c (b c-2)+2-c^2} = \frac{a \tfrac{\log(r)}{b-a} \left (a  \tfrac{\log(r)}{b-a}-2 \right )+2 - \left ( \tfrac{\log(r)}{b-a} \right )^2}{b \tfrac{\log(r)}{b-a}\left (b \tfrac{\log(r)}{b-a}-2 \right )+2 - \left ( \tfrac{\log(r)}{b-a} \right )^2} .
$$
Or in other words
\begin{align*}
r  \left ( b \log(r)\left (b \log(r)-2 (b-a) \right )+2 (b-a)^2 - \log^2(r) \right ) & \\
= a \log(r)  \left (a  \log(r)-2 (b-a) \right )&+2 (b-a)^2 - \log^2(r).
\end{align*}

At this point, we see that $b$ can be expressed as a function of $(a,r)$ as a root of a quadratic equation. Out of the two possible roots of this equation, it is easily checked that exactly one is greater than $a$. We get the explicit formula
$$
b = b(a,r) := \frac{\log (r) \left(\text{sgn}(r-1) g(a,r)-a (r+1) \right)+2 a (r-1)}{q(a,r)} ,
$$
where
\small
\begin{align*}
g(a,r) =  \sqrt{-\left(a^2-2\right) (r-1)^2+r \left(a^2+r-1\right) \log ^2(r)-2 r (r-1) \log (r)} ,
\end{align*}
\normalsize
and
$$
q(a,r) = \left(2 r+r \log ^2(r)-2 r \log (r)-2\right).
$$

Using this formula, proving equation \eqref{eq:ntsapx2} under the constraint \eqref{eq:constraint1} amounts to showing that 
\begin{equation} \label{eq:finalH}
H(a,b(a,r),c(a,b(a,r),r)) \leq 2, ~~ \forall (a,r) \in [-1,1] \times (0, \infty).
\end{equation}
Plugging the definition of $b(a,r)$ and $c(a,r)$ into \eqref{eq:eqH2} gives
$$
H(a,r) := H(a,b(a,r),c(a,b(a,r),r)) = \frac{r \left ( b(a,r)^3-3b(a,r) \right )-(a^3-3a)}{r-1}.
$$
It turns out that $H(a,r)$ is monotone decreasing in both $a$ and $r$ in the domain $[-1,1] \times (0, \infty)$ (up to a removable discontinuity on $[-1,1] \times \{1\})$. Moreover, it is straightforward to check that $\lim_{r \to 0^+} H(-1,r) = 2$. These two facts complete the proof of inequality \eqref{eq:finalH}. We omit further details of this proof.
\end{proof}

\setcounter{lemma}{3}
\begin{lemma}
Let $\eps > 0$, and $X$ a real log-concave random variable with density $\lambda$. Let $x_1 < x_0 < x_2$ be three points satisfying $\lambda(x_1) < \eps \lambda(x_0)$ and $\lambda(x_2) < \eps \lambda(x_0)$. Then, with $c(\epsilon)= \left(1+ \frac{2}{\log(1/\epsilon)} \right)^3 \left(1+ \frac{2}{\log(1/\epsilon)} + \frac{2}{\log^2(1/\epsilon)}\right)$, one has
\begin{eqnarray*}
\left (1 - 2 c(\epsilon) \epsilon \log^2(1/\epsilon) \right ) \mathrm{Var}(X) & \leq & \int_{x_1}^{x_2} (x-x_0)^2 \lambda(x) dx  \\
& \leq & \E( |X-x_0|^2 \ | \ X \in [x_1, x_2]) .
\end{eqnarray*}
\end{lemma}

\begin{proof}
We only have to prove the first inequality, as the second one is obviously true. By rescaling and translating, we can assume without loss of generality that $x_0 = 0$ and $\E(X^2) = 1$. Under these conditions we will prove the slightly stronger following bound:
\begin{equation} \label{eq:tail1}
\int_{x_1}^{x_2} x^2 \lambda(x) dx \geq 1- 2 \left(1+ \frac{2}{\log(1/\epsilon)} \right)^3 \left(1+ \frac{2}{\log(1/\epsilon)} + \frac{2}{\log^2(1/\epsilon)}\right) \epsilon \log^2(1/\epsilon) .
\end{equation}
We prove that in fact
\begin{equation} \label{eq:tail2}
\int_{x_2}^{+\infty} x^2 \lambda(x) dx \leq \left(1+ \frac{2}{\log(1/\epsilon)} \right)^3 \left(1+ \frac{2}{\log(1/\epsilon)} + \frac{2}{\log^2(1/\epsilon)}\right) \epsilon \log^2(1/\epsilon) ,
\end{equation}
and an identical computation yields the same upper bound for the integral on $(-\infty,x_1]$, which then concludes the proof of \eqref{eq:tail1}.

First note that, by log-concavity of $\lambda$, one has for any $x > x_2$,
$$\frac{\lambda(x)}{\lambda(x_2)} \leq \left( \frac{\lambda(x_2)}{\lambda(x_0)}\right)^{\frac{x-x_2}{x_2}} \leq \epsilon^{\frac{x-x_2}{x_2}} .$$
Using [Lemma 5.5 (a), \cite{LV07}] one has $\lambda(0) \leq 1$, and thus $\lambda(x_2) \leq \epsilon$, which together with the above display yields
$$\lambda(x) \leq \epsilon \exp\left( - \frac{x-x_2}{x_2} \log(1/\epsilon) \right) .$$
This directly implies that
$$\int_{x_2}^{+\infty} x^2 \lambda(x) dx \leq \left (\frac{x_2}{\log(1/\epsilon)} \right )^3 \left(1+ \frac{2}{\log(1/\epsilon)} + \frac{2}{\log^2(1/\epsilon)}\right) \epsilon \log(1/\epsilon)^2 .$$
Finally, using [Lemma 5.7, \cite{LV07}] it is easy to see that without loss of generality one can assume that $x_2 \leq \log(1/\epsilon) + 2$, and thus the above display directly yields \eqref{eq:tail2}.
\end{proof}

\setcounter{lemma}{4}
\begin{lemma}
Let $\phi \in C^2((a,b))$, and $\zeta(x) = \log \phi(x)$. Then 
$$
\phi \mbox{ is } n\mbox{-concave in } (a,b) \Leftrightarrow \zeta'' \leq - \tfrac 1 n (\zeta')^2 \mbox{ in } (a,b).
$$
\end{lemma}
\begin{proof}
Denote $\psi(x) = \phi^{1/n}(x)$ and $\xi(x) = \log \psi(x) = \tfrac 1 n \zeta(x)$. Then
$$
\xi''(x) = \frac{\psi''(x)}{\psi(x)} - \frac{\psi'(x)^2 }{\psi(x)^2} = \frac{(\phi^{1/n}(x))''}{\psi(x)} - \xi'(x)^2.
$$ 
So $\phi^{1/n}$ is concave if and only if $\xi'' \leq - (\xi')^2$ which proves the fact.
\end{proof}

\section*{Acknowledgements}
The authors would like to thank Jacob Abernethy and Elad Hazan and Bo'az Klartag for interesting discussions. We are grateful to Fran\c{c}ois Glineur for pointing out the connection with the universal barrier in the case of homogeneous convex cones.

\bibliographystyle{plainnat}
\bibliography{newbib}
\end{document}

%% file: Commands.tex
\newcommand{\inte}{\mathrm{int}}

\renewcommand{\phi}{\varphi}

\newcommand{\E}{\mathbb{E}}

\newcommand{\R}{\mathbb{R}}

\newcommand{\cK}{\mathcal{K}}

\newcommand{\cN}{\mathcal{N}}

\newcommand{\cE}{\mathcal{E}}

\def\ds1{\mathds{1}}
\renewcommand{\epsilon}{\varepsilon}
\newcommand{\eps}{\epsilon}

\newcommand{\argmax}{\mathop{\mathrm{argmax}}}
\newcommand{\argmin}{\mathop{\mathrm{argmin}}}

\newlength{\minipagewidth}
\newcommand{\beq}{\begin{equation}}
\newcommand{\eeq}{\end{equation}}

\newcommand{\beqa}{\begin{eqnarray}}
\newcommand{\eeqa}{\end{eqnarray}}

\newcommand{\beqan}{\begin{eqnarray*}}
\newcommand{\eeqan}{\end{eqnarray*}}

\def\ba#1\ea{\begin{align*}#1\end{align*}} 
\def\banum#1\eanum{\begin{align}#1\end{align}} 